\documentclass[11pt,english]{smfart}

\usepackage{dsfont,amssymb}
\usepackage[T1]{fontenc}
\usepackage{fourier}
\usepackage{baskervald}

\usepackage[text={6.5in,9in},centering]{geometry}

\usepackage[dvipsnames]{xcolor}   
\usepackage{xparse}
\usepackage{xr-hyper}
\usepackage[linktocpage=true,colorlinks=true,hyperindex,citecolor=teal,linkcolor=blue]{hyperref}

\newcommand{\inv}{^{\raisebox{.2ex}{$\scriptscriptstyle-1$}}}

\newtheorem{thm}{Theorem}[section]  
\newtheorem{prop}[thm]{Proposition}
\newtheorem{lem}[thm]{Lemma}
\newtheorem{cor}[thm]{Corollary}
\theoremstyle{definition}
\newtheorem{rem}[thm]{Remark}
\newtheorem{exms}[thm]{Example}

\newtheorem{defn}[thm]{Definition}

\numberwithin{equation}{section}
 
\author{Amartya Goswami}
\address{[1] Department of Mathematics and Applied Mathematics, University of Johannesburg,  Auckland Park Kingsway Campus, P.O. Box 524, Auckland Park 2006, South Africa; [2] National Institute for Theoretical and Computational Sciences (NITheCS), Johannesburg, South Africa.}

\email{agoswami@uj.ac.za}

\author{Danielle Kleyn}
\address{Department of Mathematical Sciences, Stellenbosch University, Private Bag X1 Matieland, 7602, South Africa}

\email{daniellem.kleyn@gmail.com}

\author{Kerry Porrill}

\address{Department of Mathematical Sciences, Stellenbosch University, Private Bag X1 Matieland, 7602, South Africa}
\email{kerry.porrill@gmail.com}

\title{Some structural aspects of the ring of arithmetical functions: prime ideals and beyond}

\date{}



\begin{document}
\begin{abstract}
The aim of these notes is to study some of the structural aspects of the ring of arithmetical
functions. We prove that this ring is neither Noetherian nor Artinian. Furthermore, we construct various
types of prime ideals. We also give an example of a semi-prime ideal that is not prime. We show that the ring of arithmetical functions has infinite Krull dimension.
\end{abstract}
\maketitle

\section{Introduction}

The study of arithmetical functions is always a central theme of number theory. One of the earliest work in \cite{B15} on arithmetical functions motivated rigorous study of these functions.
A function $f\colon \mathds{N}^+=\mathds{N}\backslash\{0\}\to \mathds{C}$ is called an \emph{arithmetical function}. According to this definition although any sequence of complex numbers is indeed an arithmetical function, however, as in \cite{HW60}, it is mentioned that ``functions $f(n)$ of the positive integer $n$ defined in a manner which expresses some arithmetical property of $n$,'' we expect these functions are very much relevant to number theoretic study. 

Let us recall some examples of arithmetical functions.
 For further studies of these functions, we refer to \cite{A76, C70, S89, M86, SS94, S83}.  Observe that any positive integer $n>1$ can be written as $n=p_1^{a_1}\cdots p_k^{a_k}$, where $p_i$'s are prime and $a_j$'s are positive integers.
 A \emph{M\"{o}bius function} $\mu$ is defined as follows:
\[\mu(n):=\begin{cases}
1, & n=1,\\
(-1)^k, & a_1=\cdots =a_k=1,\\
0, &\text{otherwise}.
\end{cases}\]
The \emph{Euler totient function} $\varphi(n)$ is defined to be the number of
positive integers not exceeding $n$ which are relatively prime to $n$, \emph{i.e.},
\[\varphi(n):=n\prod\limits_{p|n} \left(1 - \dfrac{1}{p}\right).\] \emph{Mangoldt's function} $\Lambda(n)$ is defined by
\[
\Lambda(n):=
\begin{cases}
\log p, &n=p^m\;\text{for some prime}\,p\,\text{and some}\,m\geqslant 1,\\
0,& \text{otherwise.}
\end{cases}
\]
\emph{Liouville's function} $\lambda(n)$ is defined as
\[
\lambda(n):=
\begin{cases}
1, &n=1,\\
(-1)^{a_1+\cdots+ a_k},& \text{otherwise}.
\end{cases}
\]
\emph{Ramanujan's tau function} is defined by 
\[\sum\limits_{n=1}^{\infty} \tau(n) x^n:=x\prod\limits_{j=1}^{\infty} (1-x^j)^{24},\, |x|<1.\]
The \emph{Dedekind psi function} $\psi(n)$ is defined by the formula:
\[
\psi(n):=n\prod\limits_{p|n} \left( 1+ \dfrac{1}{p}\right).
\]
The \emph{prime divisor function} $\Omega(n)$ is defined as the number of prime factors of $n$ counted with multiplicities, \textit{i.e.}, $\Omega(n):=a_1+\cdots+a_k,$
whereas, the \emph{distinct prime divisor function}  $\zeta(n)$ is defined by $\zeta(n):=k.$
A \emph{$p$-adic valuation} $\nu_p$ is defined by $\nu_p:=\max\{k\in \mathds{N}\mid p^k\,\text{divides}\, n\},$ \. where $p$ is a prime integer. The
\emph{Logarithmic function} $\log$ is another example of an arithmetic function. The \emph{Kronecker delta} function is defined for each $m\in \mathds{N}$ as follows:
\begin{equation*}
\delta_m(n) := \begin{cases}
1, & n = m, \\
0, & \text{otherwise}.
\end{cases}
\end{equation*}
Let $\mathcal{A}$ be the set of arithmetical functions. We define addition and multiplication (Dirichlet convolution) on $\mathcal{A}$ respectively by
$
(f+g)(n) := f(n)+g(n)$ and
$(f* g)(n) :=  \sum_{ij=n} f(i)g(j).\label{mul}
$   
We wish to do the structural study of the arithmetic ring in the next section, starting by reviewing some known facts about this ring.

\section{Structures of $\mathcal{A}$}

It is well-known (see  \cite{CE59}) that $(\mathcal{A}, +, *)$ is a commutative ring with identity $e$ defined by
\begin{equation*}\label{identity}
e(n) :=\begin{cases}
1, & n=1,\\ 
0, & \text{otherwise.}  
\end{cases}
\end{equation*}  
The ring $\mathcal{A}=(\mathcal{A}, +, *)$ is called the  \emph{arithmetic ring} or  the \emph{ring of arithmetical functions}. 
It is also easy to prove (see  \cite{A76}) that
an element $f\in \mathcal{A}$ is invertible if and only if $f(1)\neq 0.$ Furthermore, the set of invertible elements of $\mathcal{A}$ has the following structure.

\begin{lem}\label{mg} Let $G :=\{f\in \mathcal{A}\mid f(1)\neq 0 \}$ be the set of units of the arithmetic ring $\mathcal{A}$. Then $(G,*)$ is an abelian group. 
\end{lem} 

\begin{proof} 
First we observe that if $f,$ $g\in G,$ then $(f* g)(1)=f(1)g(1)\neq 0$, \textit{i.e.}, the set $G$ is closed under the operation $*$. The rest follows from the fact that ($\mathcal{A}, *$) is a commutative monoid, whereas existence of inverses of an $f$ follows from the hypothesis: $f(1)\neq 0.$
\end{proof}

\begin{exms}\label{innin}
From the invertibility criteria of an element in $\mathcal{A}$, it follows that the arithmetical functions $\Omega,$ $ \zeta,$ $ \Lambda,$ $ \nu_p,$ $ \log$ are non-invertible; whereas  $\mu,$ $ \varphi,$ $ \lambda,$ $\tau,$ $ \psi$ are invertible.
\end{exms}

Recall that an arithmetical function $f$ is called \emph{additive} (\emph{completely additive}) if $f(mn)=f(m)+f(n)$ for coprime positive integers (all positive integers) $m$ and $n$. We observe that if $f$ is additive or completely additive, then $$f(1)=f(1\cdot 1)=f(1)+f(1)=2f(1)$$ implying that $f(1)=0.$ Therefore, we have

\begin{prop}
If $f$ is additive or completely additive, then $f$ is non-invertible.
\end{prop}

The converse of the above statement may not be true. For example, the functions $\Lambda$ and $\log$ are non-invertible, but they are not additive (and hence not completely additive). Moreover, it is well-known that

\begin{prop}
The set $\mathrm{Add} (\mathcal{A})$ of additive arithmetical functions forms a group under pointwise addition.
\end{prop}

\begin{thm}
The arithmetic ring $\mathcal{A}$ is a local ring. 
\end{thm}	

\begin{proof}
Let $\mathfrak{m} :=\mathcal{A}\setminus G=\{f\in \mathcal{A}\mid f(1)= 0 \}.$ Since $G$ is the set of units of $\mathcal{A}$ (see Lemma \ref{mg}), by  \cite[Chapter II, \textsection 3, no.1, Proposition 1]{B72}, it is sufficient to show that $\mathfrak{m}$ is an ideal of $\mathcal{A}.$  Suppose $f$, $f' \in \mathfrak{m}$ and $g\in G.$ Then $(f+f')(1)=f(1)+f'(1)=0$ and \[(f*g)(1)=f(1)g(1)=0\cdot g(1)=0.\] Therefore $f+f'$, $f*g\in \mathfrak{m}$.  
\end{proof}

\begin{rem}
Recall that an ideal $I$ of a ring $R$ is called \emph{essential} if $I\cap J\neq \{0\}$, for every non-zero ideal $J$ of $R$. It is easy to see that $\mathfrak{m}$ is indeed an essential ideal of $\mathcal{A}$.
\end{rem}

Next, we wish to prove that  the arithmetic ring $\mathcal{A}$ is not Artinian, and for that we introduce  a norm.
A \emph{norm}  on $\mathcal{A}$ is a map $\omega\colon \mathcal{A} \to \mathds{Z}$ defined by 
\[
\omega(f):=\begin{cases}
0, & f=0;\\
n, & n\;\text{is the least positive integer with}\; f(n)\neq 0.
\end{cases}
\]  

The norm $\omega$ is not a ring homomorphism, however, $\omega$ preserves the (multiplicative) monoid structure of the arithmetic ring $\mathcal{A}$, that is,

\begin{prop}\label{np}
The norm $\omega$ is a (multiplicative) monoid homomorphism.
\end{prop}

\begin{proof}
We observe that $\omega(e)=1.$ We wish to prove
\begin{equation}\label{monhom}
\omega(f* g)=\omega(f) \omega(g), 
\end{equation}
for all $f,$ $g \in \mathcal{A}.$ If $f$ or $g$ is zero, then the condition (\ref{monhom}) is satisfied trivially. Therefore,  assume that $\omega(f)=i$ and $\omega(g)=j$, where $i$ and $j$ are positive integers. If $n=ij$, then
\begin{align*}
(f*g)(n)&=(f*g)(ij)\\&=\sum_{st=n}f(s)g(t)=f(1)g(st)+\cdots + f(i)g(j)+\cdots+f(st)g(1)\\&=f(i)g(j),
\end{align*}
where the last equality follows from the facts: (i) $f(k)=0$ for all $k<i$ and $k$ is a divisor of $n$, (ii) $g(l)=0$ for all $l<j$ and $l$ is a divisor of $n.$ From this, it follows, $\omega(f*g)=ij,$ and hence, we have the claim (\ref{monhom}).
\end{proof}

\begin{thm}
For each positive integer $n,$ the subset $$I_n :=\{ f\in \mathcal{A} \mid \omega(f)\geqslant n\}\cup \{0\}$$ is an ideal of $\mathcal{A}$, and $I_1\supseteq I_2\supseteq I_3\supseteq \cdots$ is an infinite descending chain of ideals. Therefore, $\mathcal{A}$ is not Artinian.
\end{thm}

\begin{proof}
Let $f,$ $g\in I_n.$ If $f = 0$, then $f+g = g \in I_n$, and similarly, so if $g = 0$. If $f \neq 0$ and $g \neq 0$, then $\omega(f)\geqslant n$ and $\omega(g)\geqslant n.$ Since $(f+g)(n)=f(n)+g(n),$ at once we have $\omega(f+g)\geqslant n,$ and hence, $f+g\in I_n$ if $f \neq 0$ and $g \neq 0$.   Now, let $h\in \mathcal{A}.$ To show $h*f\in I_n,$ we consider the following two possibilities: 
\begin{enumerate}
\item $h=0$ or $f=0$: Obviously $h*f=0\in I_n.$

\item $h\neq 0$ and $f \neq 0$:  
Since $f(k)=0$ for all $k\leqslant n-1,$ we also have \[(h*f)(k)=\sum_{ij=k}h(i)f(j)=0,\]
\end{enumerate}
for all $k\leqslant n-1.$ Since $\mathcal{A}$ is an integral domain, this implies there exists $k\geqslant n$ such that $\omega(h*f)=k.$
The second part of the proposition follows from the fact that $\omega(f)\geqslant n$ implies $\omega(f)\geqslant n-1$.
\end{proof}

\begin{rem}
Since  the arithmetic ring $\mathcal{A}$ is an integral domain, it follows that $\mathcal{A}$ does not have any non-trivial idempotent elements. Here is an alternative proof of this fact without using the domain property of $\mathcal{A}$.
Suppose $f$ is a non-trivial idempotent element of the arithmetic ring $\mathcal{A}$. Then we would have that $$\omega(f)^2 = \omega(f^2) = \omega(f)$$ since $f^2 = f$. This implies that $\omega(f) \in \{0, 1\}$. If $\omega(f) = 0$, then $f = 0$ and if $\omega(f) = 1$, then $f$ is a unit. In both cases, we have a contradiction as we assumed that $f$ is non-trivial.
\end{rem}

Recall from \cite{C68}
that an element of an integral domain is called an \emph{atom} or \emph{irreducible} if it is a non-unit which cannot
be written as a product of two non-units. 
If every element of a ring $R$ which is not a
unit or $0$ can be written as a product of atoms,  then $R$ is said to be \emph{atomic}. Since $\mathcal{A}$ is a unique factorization domain (see \cite{CE59}), we have the following obvious result.

\begin{thm}\label{atomic}
The arithmetic ring $\mathcal{A}$ is an atomic domain.
\end{thm}

However, note that atomicity may not be preserved by $\omega$, and this is explained in the next result.

\begin{lem}\label{irrp}
If $f$ is an atom of $\mathcal{A}$, then $\omega(f)$ may not be an atom in $\mathds{Z}$. In fact, for all $c \in \mathds{N}$, there exists an atom $f \in \mathcal{A}$ such that $\omega(f) = c$.
\end{lem}

\begin{proof}
Consider the case when $c$ is a prime number, and let $f \in \mathcal{A}$ be such that $\omega(f) = c$. For contradiction, assume that $f = g * h$, for some $g$, $h \in \mathfrak{m}$. Then, we have $$f(c) = g(1)h(c) + g(c)h(1) = (0)h(c) + g(c)(0) = 0,$$ which is a contradiction. Hence, $f$ is an atom.

Now, let $c$ be a composite integer. Consider $f \in \mathcal{A}$ such that $\omega(f)=c$ and $f(c+1)\neq 0$. Suppose that this $f$ is not an atom, so that there exists $g$, $h \in \mathfrak{m}$ with $f=g*h$. Therefore, $\omega(f)=\omega(g)\omega(h)$. Since $g$ and $h$ are non-zero non-units, it follows that $\omega(g)$, $\omega(h) > 1$. Let $a=\omega(f), b=\omega(h)$. Now consider
\begin{align*}
f(c+1)
&=(g*h)(c+1)\\
&=\sum_{ij=c+1}g(i)h(j)\\ 
&=\sum_{ij=ab+1}g(i)h(j)\\ 
&=\sum_{\substack{ij=ab+1 \\ i < a}}g(i)h(j) + \sum_{\substack{ij=ab+1 \\ i \geqslant a}}g(i)h(j) \\
&=\sum_{\substack{ij=ab+1 \\ i < a}}g(i)h(j) + \sum_{\substack{ij=ab+1 \\ j < b + \frac{1}{a}}}g(i)h(j). 
\end{align*}
Considering $ij=ab+1$, if $i<a$, then $g(i)=0$ (since $\omega(g)\geqslant a$), and so, $g(i)h(j)=0$. It cannot be the case that $i=a$, since $a\nmid ab+1$, and furthermore, if $i>a$, then \[j<\frac{ab+1}{a}=b+ \frac{1}{a}.\] Since $j$ is an integer, this implies that $j\leqslant b$, but since $b \nmid ab+1$, we have $j\neq b$. Hence $j<b$, and since $\omega(h)=b$, we can see that $h(j)=0$ so that $g(i)h(j)=0$. Therefore, $ij=ab+1$ implies that $g(i)h(j)=0$. Hence, \[\sum_{ij=ab+1}g(i)h(j) = 0,\] which is a contradiction with our assumption that $f(ab+1)\neq 0$. Hence, $f$ is an atom.
Therefore, if $\omega(f)$ is an atom, it is not necessarily true that $\omega(f)$ is a prime integer.
\end{proof}

\section{Prime ideals in $\mathcal{A}$}

Our next goal is to give various examples of prime ideals of  $\mathcal{A}$. We start with two obvious ones. Since $\mathcal{A}$ is an integral domain, the zero ideal is a prime ideal of $\mathcal{A}$. Also, the unique maximal ideal $\mathfrak{m}$ of  $\mathcal{A}$ is a prime ideal. To obtain the next example of a prime ideal of $\mathcal{A}$, we recall that an element $p$ of  $\mathcal{A}$ is called \emph{prime} if $p\mid a*b$ implies $p\mid a$ or $p\mid b$, for all $a,$ $b\in \mathcal{A}$.

\begin{prop}\label{prim}
If there exists a prime integer $p$ for which $f(p)\neq0$, then $\langle f \rangle$ is a prime ideal of $\mathcal{A}.$\end{prop}

\begin{proof}
We see that if $f = g*h$ for some $g,$ $h\in \mathfrak{m}$, then $$f(p) = g(p)h(1) + g(1)h(p) = 0,$$ which is a contradiction. Hence, if $f=g*h$ for some $g$, $h \in \mathcal{A}$, then either $g$ is a unit or $h$ is a unit, and thus, $f$ is an atom. 
Now note that $f$ is a prime element of $\mathcal{A}$ if and only if $\langle f \rangle$ is a prime ideal of $\mathcal{A}$ (see \cite[p. 98]{B70}), and since  $\mathcal{A}$ is a unique factorization domain, $f$ is prime, and hence, we have the desired claim.
\end{proof}
\begin{prop}
Let $P$ be a principal prime ideal in $\mathcal{A}$. Then $P$ is a minimal prime ideal.
\end{prop}

\begin{proof}
Let $P := \langle f \rangle$. Let $Q$ be a prime ideal such that $Q \subsetneq P$. Consider $h \in Q \subsetneq P$. Since $h \in P$, we have that there exists $g_1 \in \mathcal{A}$ such that $h = g_1 * f$. Hence, $g_1 \in Q$ or $f \in Q$ as $Q$ is a prime ideal. If $f \in Q$, then $\langle f \rangle \subseteq Q \subsetneq \langle f \rangle$ which is a contradiction. Thus $g_1 \in Q \subsetneq P$.
Following the previous argument, we then have that $g_1  = g_2 * f$, for some $g_2 \in \mathcal{A}$. Hence, $g_2 \in Q$ or $f \in Q$ as $Q$ is a prime ideal. Since $f \in Q$ is a contradiction, we have that $g_2 \in Q$.
Repeating these steps, we  arrive at the following equations for some $g_1$, $g_2$, $g_3$, $\ldots\in \mathcal{A}$:
\begin{align*}
h &= g_1 * f, \\
g_1  &= g_2 * f, \\
g_2 &= g_3 * f ,\\
g_3 &= g_4 * f, \\
&\vdots
\end{align*}
and hence, we have
\begin{align*}
h &= g_1 * f \\
&= g_2 * f^2 \\
&= g_3 *f^3 \\
&= g_4 * f^4\\
&\vdots
\end{align*}
and, in general, we have $h = g_n * f^n$, for any $n \in \mathds{N}$. Hence, we must have $\omega(h) = \omega(g_n) * \omega(f)^n$ for every $n \in \mathds{N}$. This implies that $\omega(h)=0$. Hence, $h=0$, and $Q$ is the zero ideal. Thus, $P$ is a minimal prime ideal.
\end{proof}

We now introduce two types of subsets of $\mathcal{A}$, namely $P_m$ and $J_Q$, and prove that these are prime ideals. For the rest of the paper, we shall denote the set of all prime integers by $\mathcal{P}$.

\begin{defn} \label{P_m}
Let $m \in \mathds{N}$. We define \[P_m := \left\{f \in \mathcal{A} \mid f(n) = 0 \mbox{ if } \gcd(m, n) = 1\right\}.\]
\end{defn}

\begin{rem}
We note that $P_0 = \mathfrak{m}$ and $P_1 = \{0\}$.
\end{rem}

\begin{defn}
Let $Q$ be a set of prime integers and define $\Pi_Q$ to be the set of products of a finite number of elements from $Q$. That is, \[\Pi_Q := \left\{q_1^{a_1}q_2^{a_2}\cdots q_t^{a_t} \mid t\in \mathds{N},\; q_1, q_2, \ldots, q_t \in Q,\; a_1, a_2, \ldots, a_t \in \mathds{N}\right\}.\]
We define \[J_Q := \left\{f \in \mathcal{A} \mid f(n) = 0 \text{ if } n \in \Pi_Q \right\} .\]
\end{defn}

\begin{rem}
We note that $J_{\emptyset} = \mathfrak{m}$ and $J_{\mathcal{P}} = \{0\}$.
\end{rem}

\begin{cor} \label{P_m}
Let $m \in \mathds{N}$. Then $P_m$ is an ideal.
\end{cor}
\begin{proof}
Let $f$, $ g \in P_m$. Then for any $n \in \mathds{N}$ such that $gcd(m, n) = 1$ we see that \[(f+g)(n) = f(n) + g(n) = 0,\] and hence, $f+g \in P_m$. If we have $h \in \mathcal{A}$, then we note that \[(h*f)(n) = \sum_{ij=n}h(i)f(j) = \sum_{\substack{ij=n \\ gcd(j, m) = 1}}h(i)f(j) = 0,\] and hence, $f*g \in P_m$. Thus, $P_m$ is an ideal.
\end{proof}

\begin{cor} \label{Jq-prime}
Let $Q$ be a set of prime integers. Then $J_Q$ is an ideal. 
\end{cor}
\begin{proof}
A noteworthy remark that will be useful is that if $n \in \Pi_Q$ and $k \mid n$, then $k \in \Pi_Q$. We also have that if $n_1, n_2 \in \Pi_Q$, then $n_1n_2 \in \Pi_Q$.

Let $f$, $g \in J_Q$. Then, for any $n \in \Pi_Q$, we see that $(f+g)(n) = f(n) + g(n) = 0$, and hence, $f+g \in J_Q$. If we have $h \in \mathcal{A}$, then we note that for $n \in \Pi_Q$ we have \[(h*f)(n) = \sum_{ij = n}h(i)f(j) = \sum_{ij = n}h(i)(0) = 0,\] since if $j \mid n$, then $j \in \Pi_Q$. Hence $f*g \in J_Q$ and $J_Q$ is an ideal.
\end{proof}

\begin{prop}\label{pm-sum}
Let $m\in\mathds{N}$ such that $m
\geqslant 2$, and let $q_1,\dots,q_k$ be all the prime divisors of $m$. Then $P_m=P_{q_1}+\cdots +P_{q_k}$.
\end{prop}

\begin{proof}
We start with the observation that if $m_1 \mid m_2$, then $P_{m_1} \subseteq P_{m_2}$. To prove this, consider $f \in P_{m_1}$ and consider $n \in \mathds{N}$ such that $\gcd(m_2, n) = 1$. Then, we see that $\gcd(m_1, n) \mid \gcd(m_2, n)$ so that $\gcd(m_1, n) = 1$. Hence, $f(n) = 0$ for all such $n$ so that $f \in P_{m_2}$. 
Since $q_i\mid m$ for all $i = 1, 2, \ldots, k$, we have $P_{q_i} \subseteq P_m$ for all $i = 1, 2, \ldots, k$. Hence we can conclude $P_m\supseteq P_{q_1}+\cdots +P_{q_k}$.

Now to prove $P_m \subseteq P_{q_1} + P_{q_2} + \cdots + P_{q_k}$, we apply induction on $k$. If $k=1$, then $m=q_1^{a_1}$ for some $a_1 \in \mathds{N}$. Consider $f\in P_m$ and $n \in \mathds{N}$ such that $\gcd(q_1, n)=1$. This implies that $q_1 \nmid n$ and we can also conclude that $\gcd(q_1^{a_1}, n) = 1$ so that $f(n) = 0$. Since $f(n) = 0$ for all such $n$, we see that $f \in P_{q_1}$, and hence $P_{m} \subseteq P_{q_1}$.

Now suppose for some $1\leqslant i$, we have that if $m$ has $i$ distinct prime divisors $q_1,\dots q_i$, then $P_m\subseteq P_{q_1}+\cdots +P_{q_i}$. Now, suppose $m$ has $i+1$ prime divisors. Let $q_{i+1}$ be a prime dividing $m$. We can write $m=q_{i+1}^bm^\prime$, where $b\geqslant 1$ and $\gcd(q_{i+1},m^\prime)=1$. Let $f\in P_m$. Define the following:
 \[
f_1(n) =
\begin{cases}
f(n), & q_{i+1} \nmid n, \\
0, & \text{otherwise;}
\end{cases}
\qquad \qquad
f_2(n) =
\begin{cases}
f(n), & q_{i+1} \mid n, \\
0, & \text{otherwise.}
\end{cases}
\]
Consider $n \in \mathds{N}$ such that $\gcd(m^\prime,n) = 1$. Then either $\gcd(m,n)=1$ or $\gcd(n,m)=q_{i+1}^c$ for some $c\in \mathds{N}$. If $\gcd(m, n) = 1$, then $f(n)=0$ so that $f_1(n)=0$. Alternatively, if $\gcd(n,m)=q_{i+1}^c$ for some $c\in \mathds{N}$, then we can conclude that $q_{i+1}\mid n$. Hence, $f_1(n)=0$ by definition of $f_1$. Now we have that for all $n\in \mathds{N}$, if $\gcd(m^\prime,n)=1$, then $f_1(n)=0$. We can conclude that $f_1\in P_{m^\prime}$.
Now, consider $f_2$.  Notice that for all $n\in\mathds{N}$, if $\gcd(q_{i+1},n)=1$ then $q_{i+1}\nmid n$ so that $f_2(n)=0$. Thus, $f_2\in P_{q_{i+1}}$. Now note that \[f_1(n)+f_2(n)=\begin{cases}
f(n), & q_{i+1} \mid n, \\
f(n), & \text{otherwise},
\end{cases}\] so that $f_1+f_2=f$. Hence, we can conclude that $P_m\subseteq P_{m^\prime}+P_{q_{i+1}}$. By the induction hypothesis, we have that $P_m\subseteq P_{q_1}+\cdots +P_{q_i}+P_{q_{i+1}}$, where $q_1,\dots, q_i$ are the prime factors of $m^\prime$. Thus, by induction, we have that $P_m \subseteq P_{q_1} + P_{q_2} + \cdots + P_{q_k}$. Now we can conclude that $P_m = P_{q_1} + P_{q_2} + \cdots + P_{q_k}$.
\end{proof}

\begin{cor}\label{Pm-JQ}
Let $Q$ be a set of prime integers. If $\mathcal{P} \backslash Q$ is finite, then $J_Q = P_m$, for some $m \in \mathds{N}$. Specifically, if $\mathcal{P} \backslash Q = \{q_1, q_2, \ldots, q_k\}$ for some $k\in \mathds{N}$, then $m = q_1q_2\cdots q_k$. 
\end{cor}

\begin{proof}
Consider $f \in J_Q$ and $n \in \mathds{N}$ such that $\gcd(m, n) = 1$. This means that $q_i \nmid n$ for all $i \in \{1, 2, \ldots, k\}$. Therefore, if $n$ has a prime factorisation of $n = r_1^{a_1}r_2^{a_2}\cdots r_t^{a_t}$ for some $t \in \mathds{N}$, then $r_i \notin \{q_1, q_2, \ldots, q_k\} = \mathcal{P} \backslash Q$ for all $i \in \{1, 2, \ldots, t\}$. Hence, $r_i \in Q$ for all $i \in \{1, 2, \ldots, t\}$ so that $n \in \Pi_Q$. Now, since $f \in J_Q$, we can see that, $f(n) = 0$, which lets us conclude that $f \in P_m$. Hence, $J_Q \subseteq P_m$.
Now consider $f \in P_m$ and $n \in \Pi_Q$. Now $\gcd(m, n) = \gcd(q_1q_2\cdots q_k, n) \neq 1$ if and only if $q_i \nmid n$ for all $i\in \{1, 2, \ldots, k\}$. If $q_i \mid n$ for some $i \in \{1, 2, \ldots, k\}$, then $n \notin \Pi_Q$ by definition, which is a contradiction. Hence $\gcd(m, n) = 1$ and $f(n) = 0$, so that, $f \in J_Q$. Hence, $P_m \subseteq J_Q$ so that $J_Q = P_m$.
\end{proof}

\begin{cor} \label{same prime ideal}
Let $m_1, m_2 \in \mathds{N}$. Then $P_{m_1} = P_{m_2}$ if and only if $m_1$ and $m_2$ share all of their prime divisors.
\end{cor}

\begin{proof}
We see that if $m_1$ and $m_2$ share all of their prime divisors, then $P_{m_1} = P_{m_2}$ using Proposition \ref{pm-sum}.
Now say $P_{m_1} = P_{m_2}$, and for contradiction assume that $m_1$ and $m_2$ do not share all of their prime divisors. Let $r$ be a prime such that $r \mid m_1$ and $r \nmid m_2$. Now consider $\delta_r \in \mathcal{A}$. Then $\delta_r \in P_{m_1}$ and $\delta_r \notin P_{m_2}$ as we will have $\gcd(n, m_1) = 0$ for all $n \in \mathds{N}\backslash\{r\}$ (with $\gcd(r,m_1) = r \neq 1$) and $\gcd(r, m_2) = 1$ but $\gcd(r, m_2) = 1 \neq 0$. This is a contradiction, as we assumed $P_{m_1} = P_{m_2}$. Thus, $m_1$ and $m_2$ must share all of their prime divisors.
\end{proof}

\begin{prop} \label{Jq-prime}
Let $Q$ be a set of prime integers. Then $J_Q$ is a prime ideal. 
\end{prop}

\begin{proof}
To show that $J_Q$ is a prime ideal, consider $f^\prime$ and $g^\prime \in \mathcal{A} \backslash J_Q$. Since $f^\prime \notin J_Q$, there exists a smallest $n_1 \in \Pi_Q$ such that $f^\prime(n_1) \neq 0$. Similarly, there exists a smallest $n_2 \in \mathds{N}$ such that $g^\prime(n_2) \neq 0$. We can thus say that 
\begin{align*}
(f^\prime * g^\prime)(n_1n_2)
& = \sum_{ij=n_1n_2}f^\prime(i)g^\prime(j) \\
& = \sum_{\substack{ij=n_1n_2 \\ i \leqslant n_1}}f^\prime(i)g^\prime(j) +  \sum_{\substack{ij=n_1n_2 \\ i > n_1}}f^\prime(i)g^\prime(j) \\
& = \sum_{\substack{ij=n_1n_2 \\ i \leqslant n_1}}f^\prime(i)g^\prime(j) +  \sum_{\substack{ij=n_1n_2 \\ j < n_2}}f^\prime(i)g^\prime(j) \\
& = f^\prime(n_1)g^\prime(n_2) + 0 \\
& = f^\prime(n_1)g^\prime(n_2)\\
& \neq 0,
\end{align*}
which implies that $f^\prime * g^\prime \notin J_Q$. Thus, by the contrapositive argument, $J_Q$ is a prime ideal as we proved equivalently that $f*g \in J_Q$ must imply that $f \in J_Q$ or $g \in J_Q$.
\end{proof}

\begin{cor}
Let $m \in \mathds{N}$. Then $P_m$ is a prime ideal.
\end{cor}
\begin{proof}
We have shown in Corollary \ref{Pm-JQ} that $P_m=J_Q$ for the correct choice of $Q$. Hence, by Proposition \ref{Jq-prime}, we see that $P_m$ is prime.  
\end{proof}

\begin{prop}\label{ppid} 
Let $p$ be a prime integer. Then $P_p$ is a principal ideal. Specifically $P_p=\langle \delta_{p} \rangle$.
\end{prop}

\begin{proof}
Note that if $p$ is prime, then $\gcd(n,p) = 1$ is equivalent to $p \nmid n$. Hence \[P_p := \left\{f \in \mathcal{A} \mid f(n) = 0 \mbox{ if } p\nmid n \right\}.\]  
Let $f \in P_p$. Define $g$ as follows, $g(n):=f(np)$. Since $f \in P_p$, $f(n)=0$ when $p \nmid n$ and by definition of $g$ we have $f(n)=g\left(\frac{n}{p}\right)$ if $p\mid n$. Then, we have that \begin{align*}
(\delta_p*g)(n) 
&= \sum_{ij=n}\delta_p(i)g(j) \\
&= \sum_{\substack{ij=n \\ i = p}}g(j) \\
&=\begin{cases}
0, & p\nmid n;\\
g\left(\frac{n}{p}\right), & \text{otherwise}
\end{cases}\\
&=f(n). 
\end{align*}
Hence, $f \in \langle \delta_p \rangle$ so that $P_p \subseteq \langle \delta_p \rangle$.
Now we note that $\delta_p \in P_p$, which can easily be checked by definition of $P_p$. Hence, $\langle \delta_p \rangle \subseteq P_p$ so that $\langle \delta_{p}\rangle=P_p$, and $P_p$ are principal ideals.
\end{proof}
\begin{cor} \label{P_m generators}
Let $m\in\mathds{N}$ such that $m
\geqslant 2$, and let $q_1,\dots,q_k$ be all the prime divisors of $m$. Then $P_m = \langle \delta_{q_1}, \dots, \delta_{q_k} \rangle$.
\end{cor}

\begin{proof}
By Proposition \ref{pm-sum}, $P_m=P_{q_1}+\cdots +P_{q_k}$. Furthermore, for each $1\leqslant i\leqslant k$, we have by Proposition \ref{ppid} that $P_{q_i}=\langle \delta_{q_i}\rangle$. Thus, we can conclude that $P_m=\langle \delta_{q_1}, \dots, \delta_{q_k} \rangle$.
\end{proof}

\begin{lem} \label{ideal_num_generators}
Let $I$ be a proper ideal of $\mathcal{A}$ and let $k \geqslant 1$. Assume that there exist distinct primes
$q_1, \ldots, q_k$ such that the map
\begin{align*}
\varphi\colon & I\to \mathds{C}^k\\
&f\mapsto [f(q_1), \ldots, f(q_k)]^T
\end{align*}
is surjective. Then, $I$ requires at least $k$ generators.
\end{lem}

\begin{proof}
Assume for contradiction that $I$ is generated by $k^*$ elements, say $I = \langle f_1, f_2, \ldots, f_{k^*}\rangle$, where $k^* < k$. Let $f$ be an arbitrary element of $I$. Then $$f = f_1*g_1 + f_2*g_2 + \cdots + f_{k^*}*g_{k^*},$$ for some $g_1$, $g_2$, $\ldots$, $g_{k^*} \in \mathcal{A}$, and for all $r \in \{1, 2, \ldots, k\}$, we have
\begin{align*}
f(q_r) &= \sum\limits_{i=1}^n \left(f_i(q_r)g_i(1) + f_i(1)g_i(q_r)\right) \\
&= \sum\limits_{i=1}^n \left(f_i(q_r)g_i(1) + 0\cdot g_i(q_r) \right)\\
&= \sum\limits_{i=1}^n f_i(q_r)g_i(1).
\end{align*}
Now, the map $\varphi: I \to \mathds{C}^k$ can be rewritten as 
\begin{align*}
\varphi(f) :=
\begin{bmatrix}
f(q_1) \\
f(q_2) \\
\vdots \\
f(q_k)
\end{bmatrix}
&= g_1(1)
\begin{bmatrix}
f_1(q_1) \\
f_1(q_2) \\
\vdots \\
f_1(q_k)
\end{bmatrix}
+ g_2(1)
\begin{bmatrix}
f_2(q_1) \\
f_2(q_2) \\
\vdots \\
f_2(q_k)
\end{bmatrix}
+ \cdots +
g_{k^*}(1)
\begin{bmatrix}
f_{k^*}(q_1) \\
f_{k^*}(q_2) \\
\vdots \\
f_{k^*}(q_k)
\end{bmatrix}.
\end{align*}
We then see that $\left\{\varphi(f)\mid f \in I\right\}$ is at most $k^*$-dimensional over $\mathds{C}$ since every $\varphi(f)$ can be written as a linear combination of the $n$ vectors \[\left[f_1(q_1), f_1(q_2), \ldots, f_1(q_k)\right]^T, \ldots, \left[f_{k^*}(q_1), f_{k^*}(q_2), \ldots, f_{k^*}(q_k)\right]^T.\] Since $\left\{\varphi(f)\mid f\in I\right\} = \mathds{C}^{k}$, we will have that $\mathds{C}^{k}$ is at most $k^*$-dimensional over $\mathds{C}$. However, this is a contradiction since $\mathds{C}^{k}$ is $k$-dimensional over $\mathds{C}$ and $k^* < k$.
Hence, $I$ cannot be generated by fewer than $k$ elements.
\end{proof}

\begin{prop} \label{P_m num_generators}
Let $m\in\mathds{N}$ such that $m \geqslant 2$, and let $q_1,\dots,q_k$ be the distinct prime divisors of $m$. Then $P_m$ has exactly $k$ generators.
\end{prop}

\begin{proof}
Using Corollary \ref{P_m generators}, we already know that $P_m = \langle \delta_{q_1}, \ldots, \delta_{q_k}\rangle$. 
To see that $P_m$ cannot have fewer than $k$ generators, we note that the map 
\begin{align*}
\varphi\colon\, & P_m\to \mathds{C}^k\\
&f\mapsto [f(q_1), \ldots, f(q_k)]^T
\end{align*}
is surjective since no restriction is defined on these values in the definition of $P_m$. Using Lemma \ref{ideal_num_generators}, we then know that $P_m$ cannot be generated by fewer than $k$ elements.
\end{proof}
\begin{cor}\label{JQ_infinitely_generated}
Let $Q$ be a set of prime integers. If $\mathcal{P} \backslash Q$ is an infinite set, then $J_Q$ is not finitely generated.
\end{cor}
\begin{proof}
Let us assume that $J_Q$ is finitely generated. Say $J_Q = \langle f_1, f_2, \ldots, f_n \rangle$ for some $f_1$, $f_2$, $\ldots$, $f_n \in \mathcal{A}$. Select $p_1$, $p_2$, $\ldots$, $p_{n+1} \in \mathcal{P} \backslash Q$ such that $p_1 < p_2 < \cdots < p_{n+1}$. Note, this is possible to do since we are assuming that $\mathcal{P} \backslash Q$ is infinite.
Now we can define a map $\varphi: J_Q \to \mathds{C}^{n+1}$ as follows
\begin{align*}
\varphi\colon & J_Q\to \mathds{C}^{n+1}\\
&f\mapsto 
 \left[f(p_1), \ldots, f(p_{n+1})\right]^T.
\end{align*}
In fact, $\varphi$ is surjective since $J_Q$ does not put any condition on the value of $f(p)$ for every prime $p \in \mathcal{P} \backslash Q$ and $f \in J_Q$. By Lemma \ref{ideal_num_generators}, this means that $J_Q$ cannot be generated by fewer than $n+1$ generators, which is a contradiction by our assumption that $J_Q$ is generated by $n$ generators.
Hence, $J_Q$ cannot be finitely generated if $\mathcal{P} \backslash Q$ is infinite.
\end{proof}

\begin{cor}
$\mathcal{A}$ is not a B\'ezout domain.
\end{cor}

\begin{proof}
Consider $P_6$. This is generated by $\delta_{2}$, and $\delta_{3}$ and cannot be generated by $1$ element. Therefore, not every finitely generated ideal is principal.
\end{proof}

\begin{cor} \label{prime_subsets}
Let $m \in\mathds{N}$ and let $Q$ be a set of prime integers such that $r \nmid m$ for all $r \in Q$. Consider primes $p$, $q$ such that $p \mid m$ and $q \in Q$. Then \[P_p \subseteq P_m \subseteq J_Q \subseteq J_q,\] where $J_q = J_{\{q\}}$.
\end{cor}

\begin{proof}
Using Proposition \ref{pm-sum}, we see that $P_p \subseteq P_m$.
Consider $f \in P_m$ such that $f(n) = 0$ if $\gcd(m, n)=1$. Now consider $n \in \Pi_Q$. If $r \mid n$ for some prime integer $r$, then $r \in \Pi_Q$. However, $r \nmid m$ for all $r \nmid Q$ so that $\gcd(m, n) = 1$, and hence, $f(n) = 0$ for all $n \in \Pi_Q$. Hence, $f \in J_Q$ so that $P_m \subseteq J_Q$.
Let $f \in J_Q$. Consider $n \in \Pi_{\{q\}} \subseteq \Pi_Q$. Then $f(n) = 0$ since $n \in \Pi_Q$. Hence, $f \in J_q$ so that $J_Q \subseteq J_q$. 
\end{proof}

\begin{prop}\label{nnd}
Let $\pi_k$ denote the $k^{th}$ prime in $\mathds{N}$, where $\pi_1 = 2$. For each positive integer $n,$ the subset $$K_n :=\left\{ f\in \mathcal{A} \mid f(1) = 0 \mbox{ and } f(\pi_k) = 0 \mbox{ for all } k \geqslant  n \right\}$$ is a (non-prime) ideal of $\mathcal{A}$, and $K_1\subseteq K_2\subseteq K_3\subseteq \cdots$ is an infinite ascending chain of ideals. Therefore, $\mathcal{A}$ is not Noetherian.
\end{prop}

\begin{proof}
Let $f$, $g \in K_n$. Then, for all $k \geqslant n$ we have that $f(\pi_k) = 0 = g(\pi_k)$ so that \[(f+g)(\pi_k) = f(\pi_k) + g(\pi_k) = 0,\] and we have $(f+g)(1) = f(1) + g(1) = 0$. Hence $f + g \in K_n$. Now, let $h\in \mathcal{A}.$ To show $h*f\in K_n,$ we see that for any $k \geqslant  n$,
\begin{align*}
(h*f)(\pi_k)&=\sum_{ij=\pi_k}h(i)f(j)= h(1)f(\pi_k) + h(\pi_k)f(1) = 0,\; \text{and}\\
(h*f)(1) &= h(1)f(1) = 0.
\end{align*}
In addition to this, we see that $K_n$ is not prime, as we can select $f = g = (0, 1, 1, 1, 1, \ldots)$ such that \[(f*g)(p) = f(1)g(p) + f(p)g(1) = 0,\] for all primes $p$ (including primes $p \geqslant \pi_n$) and thus $f*g \in K_n$, but $f$, $g \notin K_n$. 
\end{proof}

\begin{prop}\label{kdi}
The Krull dimension of $\mathcal{A}$ is infinite.
\end{prop}

\begin{proof}
In addition to being able to construct an infinite ascending chain of non-prime ideals as in proposition \ref{nnd}, from proposition \ref{prime_subsets} and corollary \ref{same prime ideal} we see that we can construct an infinite ascending chain of prime ideals such that each one is a proper subset of the next, in particular \[P_{\pi_1} \subsetneq P_{\pi_1\pi_2} \subsetneq P_{\pi_1\pi_2\pi_3} \subsetneq \cdots,\] where $\pi_k$ refers to the $k^{th}$ prime in $\mathds{N}$. These ideals are properly contained in one another since if $m \geqslant 2$ and $q$ is a prime not dividing $m$, then we have $\delta_q \in P_{mq} \backslash P_m$. We can also construct an infinite descending chain of prime ideals, in particular \[J_{\{\pi_1\}} \supseteq J_{\{\pi_1, \pi_2\}}  \supseteq J_{\{\pi_1, \pi_2, \pi_3\}}  \supseteq \cdots.\]  
Thus, the Krull dimension of $\mathcal{A}$ is infinite.
\end{proof}

\begin{rem}
Note that the ring $\mathcal{A}$ is not Noetherian now follows from the above result. However, the independent proof of it in Proposition \ref{nnd} gives the construction of an infinite chain of ideals different from in Proposition \ref{kdi}.
\end{rem}

\begin{rem}
It is easy to see that if $f$ is invertible, then $\langle f \rangle$ is not a prime ideal. Therefore, the set of principal prime ideals is a (proper) subset of the set of non-invertible elements of $\mathcal{A}$. Using Corollary \ref{JQ_infinitely_generated}, we can note that the unique maximal ideal $\mathfrak{m}$ of $\mathcal{A}$ is not finitely generated since $J_{\emptyset} =\mathfrak{m}$.
\end{rem}







We finish this section with an example of a semi-prime ideal of $\mathcal{A}$ that is not in general a prime ideal.

\begin{prop} \label{P_mk}
Let $m, k \in \mathds{N}$ with $\mu(m) \neq 0$. Then \[P_{m, k} := \left\{f \in \mathcal{A} \mid f(n) = 0 \mbox{ if } \zeta(\gcd(m, n)) \leqslant k \right\}\] is a semi-prime ideal. 
\end{prop}

\begin{proof}
Let $f$, $g \in P_{m,k}$. Then, for any $n \in \mathds{N}$ such that $\zeta(gcd(m,n)) \leqslant k$ we see that $(f+g)(n) = f(n) + g(n) = 0$ and hence, $f+g \in P_{m, k}$. If we have $h \in \mathcal{A}$, then we can note that \[(h*f)(n) = \sum_{\substack{ij=n}}h(i)f(j) = \sum_{\substack{ij=n \\ \zeta(\gcd(j, m)) \leqslant k}}h(i)f(j) = 0,\] and hence $f*g \in P_{m ,k}$. Thus, $P_{m, k}$ is an ideal of $\mathcal{A}$.
	
To show $P_{m, k}$ is a semi-prime, consider $f \in \mathcal{A} \backslash P_{m, k}$. We will show that $f^r \in \mathcal{A} \backslash P_{m, k}$ for all $r \in \mathds{N}\backslash\{0\}$. 
Since $f \notin P_{m, k}$, we see that there must exist a smallest $n \in \mathds{N}$ with $\zeta(gcd(m, n)) \leqslant k$ such that $f(n) \neq 0$. As an induction hypothesis, assume that for some $r\in \mathds{N}$, we have that $f^s \notin P_{m, k}$ for all $s\leqslant r$. Also, assume that for all $s\leqslant r$, we have that $n^s$ is the smallest positive integer with $\zeta(\gcd(m, n^s)) \leqslant k$ and $f^s(n^s) \neq 0$. We have already established that this is the case for $r=1$. We will now show that it must also be true for $r+1$.
	
We see that if $j\mid n$, then $\gcd(m, j) \mid \gcd(m, n)$ so that $\zeta(m, j) \leqslant \zeta(m, n)$. So, if $j\mid n$ and $j\neq n$, then $f(j)=0$. Since $m$ is square-free, we also have that $\gcd(m, n) = \gcd(m, n^{r+1})$ and if $i \mid n^{r+1}$, then \[\gcd(m,i) \mid \gcd(m, n^{r+1}) = \gcd(m, n)\] so that $\zeta(gcd(m, i)) \leqslant \zeta(\gcd(m, n)) \leqslant k$. Hence, if $i\mid n^{r+1}$ and $i < n^r$, then $f^r(i)=0$. With this, we see that
\begin{align*}
(f^{r+1})(n^{r+1}) &= \sum_{ij=n^{r+1}}f^r(i)f(j) \\ 
&= \sum_{\substack{ij=n^{r+1}\\ i \leqslant n^r}}f^r(i)f(j) + \sum_{\substack{ij=n^{r+1}\\ i > n^r}}f^r(i)f(j)\\
&=\sum_{\substack{ij=n^{r+1}\\ i \leqslant n^r}}f^r(i)f(j) + \sum_{\substack{ij=n^{r+1}\\j<n}}f^r(i)f(j)\\
&=f^r(n^r)f(n) + \sum_{\substack{ij=n^{r+1}\\ i < n^r \\ \zeta(gcd(m, i)) \leqslant k}}f^r(i)f(j) + \sum_{\substack{ij=n^{r+1}\\ j<n \\ \zeta(gcd(m, j)) \leqslant k}}f^r(i)f(j)\\
&= f^r(n^r)f(n)\\
&\neq 0.
\end{align*}
Hence, $f^{r+1} \notin P_{m, k}$. 
Now consider $N \in \mathds{N}$ such that $N < n^{r+1}$ and $\zeta(gcd(m, N)) \leqslant k$. We will show that $f^{r+1}(N) = 0$ so that $n^{r+1}$ is the smallest positive integer for which $\zeta(gcd(m, n^{r+1})) \leqslant k$ and $f^{r+1}(n^{r+1}) \neq 0$. Indeed:	
\begin{align*}(f^{r+1})(N) &= \sum_{ij=N}f^r(i)f(j) \\ &= \sum_{\substack{ij=N\\ i \leqslant n^r}}f^r(i)f(j) + \sum_{\substack{ij=N\\ i > n^r}}f^r(i)f(j)\\
&=\sum_{\substack{ij=N\\ i \leqslant n^r}}f^r(i)f(j) + \sum_{\substack{ij=N\\j<\frac{N}{n^r}}}f^r(i)f(j)\\
&= \sum_{\substack{ij=N\\ i \leqslant n^r \\ \zeta(gcd(m, i)) \leq k}}f^r(i)f(j) + \sum_{\substack{ij=N\\ j<\frac{N}{n^r}<n \\ \zeta(gcd(m, j)) \leqslant k}}f^r(i)f(j)\\
&= 0.
\end{align*}
Hence $(f^{r+1})(N) = 0$ for all $N \in \mathds{N}$ satisfying $\zeta(\gcd(m, N)) \leqslant k$ and $N<n^{r+1}$. 
This proves our induction hypothesis, and as a result, $f \notin P_{m, k}$ implies that $f^r \notin P_{m, k}$ for all $r\in \mathds{N}$. Hence, $P_{m, k}$ is a semi-prime ideal.
\end{proof}

\begin{cor}
Let $m, k \in \mathds{N}$ with $\mu(m)\neq0$ and $1 \leqslant k < \zeta(m)$. Then $P_{m, k}$ is not a prime ideal.
\end{cor}

\begin{proof}
Since $\mu(m) \neq 0$, $m$ is square-free and its prime factorisation is of the form $q_1q_2\cdots q_t$ with $t = \zeta(m)$. Let $\alpha = q_1q_2\cdots q_k$ and $\beta = q_{k+1}$ and then define $a = \delta_{\alpha}$ and $b = \delta_{\beta}$. Then $ab = \delta_{\alpha\beta}$ so that $ab \in P_{m, k}$ but $a \notin P_{m, k}$ and $b \notin P_{m, k}$. Hence, $P_{m, k}$ is not a prime ideal.
\end{proof}

\begin{rem}
We see that when $k = 0$ we have that $P_{m, 0} = P_m$, which is a prime ideal. If $k \geqslant \zeta(m)$, then $\zeta(gcd(m, n)) \leqslant k$ is true for all $n \in \mathds{N}$ so that we obtain $P_{m, k} = \{0\}$.
\end{rem}

\section*{Acknowledgement}

We express our sincere gratitude to the anonymous referee for their insightful suggestions and constructive comments, which have substantially enhanced the quality and clarity of this paper.

\end{document}